\documentclass[11pt,reqno]{amsart}
\usepackage{amscd,amssymb,amsmath,amsthm}
\usepackage[arrow,matrix]{xy}
\usepackage{graphicx}
\usepackage{epstopdf}
\usepackage{color}
\newtheorem{thm}{Theorem}
\newtheorem{defn}{Definition}
\newtheorem{lemma}{Lemma}
\newtheorem{pro}{Proposition}
\newtheorem{rk}{Remark}

\numberwithin{equation}{section} 

\begin{document}

\vspace{0.5in}
\renewcommand{\bf}{\bfseries}
\renewcommand{\sc}{\scshape}
\vspace{0.5in}

\title[On ocean ecosystem discrete time dynamics]{On ocean ecosystem discrete time dynamics generated by $\ell$-Volterra operators}

\author{U. A. Rozikov, S.K. Shoyimardonov}

 \address{U.\ A.\ Rozikov\\ Institute of mathematics,
81, Mirzo Ulug'bek str., 100125, Tashkent, Uzbekistan.}
\email {rozikovu@yandex.ru}

\address{Department of Higher Mathematics; Tashkent University of Information Technologies named after Muhammad al-Khwarizmi;
Tashkent, Uzbekistan.}
\email{shoyimardonov@inbox.ru}






\begin{abstract} We consider a discrete-time dynamical system
generated by a nonlinear operator (with four real parameters $a,b,c,d$) of ocean ecosystem.
We find conditions on the parameters under which the operator is reduced to a $\ell$-Volterra
quadratic stochastic operator mapping two-dimensional simplex to itself. We
show that if $cd(c+d)=0$ then (under some conditions on $a,b$) this $\ell$-Volterra operator may have
up to three or a countable set of fixed points;
if $cd(c+d)\ne 0$ then the operator has up to three  fixed points. Depending on the parameters the
fixed points may be attracting, repelling or saddle points.  The limit behaviors of trajectories of
the dynamical system are studied. It is shown that independently on values of parameters and on
initial (starting) point all trajectories converge. Thus the operator (dynamical system) is regular.
We give some biological interpretations of our results.
\end{abstract}

\maketitle

\section{Introduction}

 An ecosystem is a community made up of living organisms and nonliving
 components such as air, water and mineral soil\footnote{https://en.wikipedia.org/wiki/Ecosystem}.

In ecosystems the energy flows are somewhat difficult to measure and to model.
The most successful ecosystem models have concentrated on the balance of
essential elements such as \emph{carbon}, \emph{nitrogen} and \emph{phosphorus}
rather than explicitly on energy \cite{Brit}, \cite{M}.

Following \cite{Brit}, \cite{M} we consider an example of plankton in ocean limited by the essential element nitrogen,
and assume that the system is closed to nitrogen. Plankton are of two kinds,
\emph{phytoplankton}, or plant plankton, which photosynthesis and require essential elements,
and \emph{zooplankton}, or animal plankton, which feed on phytoplankton.
It is assumed that all the action takes place in a well-mixed surface layer of the ocean.

 Let $N$ be the concentration of nitrogen available for uptake,
 measured as mass per unit surface area of the ocean, $P$ the concentration of phytoplankton,
 and $Z$ the concentration of zooplankton, both measured in the same currency,
 i.e. as mass of nitrogen incorporated in the plankton per unit surface area of the ocean.

Nitrogen (as dissolved gas or compounds) is taken up from the ocean
and incorporated into phytoplankton. It is incorporated into zooplankton through
consumption of phytoplankton. It is recycled from the plankton of the ocean
through death and excretion.

In \cite{A} the authors developed moment closure approximations
to represent micro-scale spatial variability in the concentrations of
nutrients, phytoplankton and zooplankton in an NPZ model.
For the NPZ closure model  the following are showed: the stability domains increases
with micro-scale variability,  increases the biomass of zooplankton, and
 the coefficient of variation of phytoplankton increases with micro-scale variability.

At time moment $t\geq 0$ the state of the ecosystem is given by the vector $(N(t), P(t), Z(t))$.

In \cite{Brit} the following model of ocean ecosystem processes is given:
\begin{equation}\label{1}
\begin{cases}
\frac{dN}{dt}=aP+bZ-cNP\\
\frac{dP}{dt}=cNP-dPZ-aP \\
\frac{dZ}{dt}=dPZ-bZ .
\end{cases}
\end{equation}
where $a,b,c,d\in R$,  it follows from the system that
$$\frac{d}{dt}(N+P+Z)=0, \   \  \  \ \mbox{so that} \ \ \ N+P+Z={\rm const},$$
i.e., this is law of
conservation of mass for nitrogen and this constant represents the
total concentration of nitrogen, both available for uptake (free)
and incorporated in the plankton (bound).

In this paper we study the discrete time version of (\ref{1}).
This system is a dynamical system generated by a 2-Volterra
 quadratic stochastic operator.

Let us first give necessary definitions, then explain what is the main problem;
secondly we give the history of its solutions and then formulate the part of the
problem which we want to solve in this paper.

\emph{The quadratic stochastic operator} (QSO) \cite{UR}  is a mapping of the simplex.
\begin{equation}\label{2}
S^{m-1}=\{x=(x_{1},...,x_{m})\in R^{m}: x_{i}\geq0, \sum\limits_{i=1}^{n} x_{i}=1\}
\end{equation}
into itself, of the form
\begin{equation}\label{3}
V:x'_{k}=\sum\limits_{i,j=1}^{m}P_{ij,k}x_{i}x_{j}, \ \ \ \ \ \  k=1,...,m,
\end{equation}
where the coefficients $P_{ij,k}$ satisfy the following conditions
\begin{equation}\label{4}
P_{ij,k}\geq 0,\ \  P_{ij,k}=P_{ji,k}, \ \   \sum\limits_{k=1}^{m}P_{ij,k}=1,  \ \ \ \  (i,j,k=1,...,m)
\end{equation}
Thus, each quadratic stochastic operator $V$ can be uniquely defined by a cubic matrix $\mathbb P=(P_{ij,k})_{i,j,k=1}^{n}$ with conditions (\ref{4}).

For a given $\lambda^{(0)}\in S^{m-1}$ the \emph{trajectory} (orbit)
$$ \{\lambda^{(n)}\}, \ \ \ \  n=0,1,2,... \text{of} \ \ \lambda^{(0)}$$
under the action of QSO (\ref{3}) is defined by

\begin{center}
$\lambda^{(n+1)}=V(\lambda^{(n)}),$   \ \ \ \ where \ \ $n=0,1,2,...$
\end{center}

One of the main problem in mathematical biology consists in the study of the
asymptotical behavior of the trajectories. The difficulty of the problem depends on given matrix $\mathbb P$.
One of simple cases is Volterra QSO, i.e., the matrix $\mathbb P$ satisfies
\begin{equation}\label{v}
P_{ij,k}=0 \ \ {\rm if} \, k\notin\{i,j\} \ \ \mbox{for any} \ \ i,j\in E.
\end{equation}
In \cite{RGan} the theory for Volterra QSO was developed by using the theories of the Lyapunov function
and of tournaments. But non-Volterra QSO's (i.e., not satisfying condition (\ref{v})) were not exhaustively
studied, because there is no general theory that can be applied to the study of non-Volterra operators.
There are a few papers devoted to such operators see, for example, \cite{GMR}, \cite{UR}-\cite{UU}.
The operator which we are going to study in this paper is another example of non-Volterra operators.

\emph{$\ell$-Volterra QSO}. Fix $\ell\in\{1,...,m\}$ and assume that elements $P_{ij,k}$ of the matrix $\mathbb P$ satisfy

\begin{equation}\label{7}\begin{array}{ll}
 P_{ij,k}=0 \ \ {\rm if} \ \ k\notin\{i,j\} \ \ \mbox{for any} \ \ k\in\{1,...,\ell\}, i,j\in E;\\[3mm]
P_{ij,k}>0\ \ \mbox{for at least one pair} \ \ (i,j), \, i\neq k, \ \ j\neq k \, \mbox{for any} \ \  k\in \{\ell+1,...,m\}.
\end{array}
\end{equation}
\begin{defn}\label{d1} \cite{UZ} For any fixed $\ell\in\{1,...,m\}$, the QSO defined by (\ref{3}), (\ref{4}) and (\ref{7}) is called $\ell$-Volterra QSO.
\end{defn}
\begin{defn}
A QSO $V$ is called regular if for any initial point $\lambda^{(0)} \in S^{m-1}$, the limit
\[\lim_{n\to \infty}V^n(\lambda^{(0)}) \]
exists.
\end{defn}

We consider a model of discrete time process of ocean
ecosystem (\ref{1}), which has the following form
\begin{equation}\label{b3}
V:
\begin{cases}
x^{(1)}=x(1-b+dy)\\
y^{(1)}=y(1-a-dx+cz)\\
z^{(1)}=z(1-cy)+ay+bx.
\end{cases}
\end{equation}
where $x=Z, y=P, z=N$, $a,b,c,d\in R$.

The set of limit points of trajectory is important in the theory of dynamical systems.
The \textbf{main problem} of our investigation is to study the
set of  limit points of trajectories of the operator (\ref{b3}).

A fixed point $p$ for a mapping $F: R^{m}\rightarrow R^{m}$ is a solution to the equation $F(p)=p$.
By the continuity of the operator $V$ (\ref{b3}) its limit points are fixed points for the operator $V$.

The paper is organized as follows. In the next Section 2
 we reduce our operator to a 2-Volterra operator mapping $S^2$ to itself
 and compare it with known 2-Volterra operators.
 Section 3 is devoted to dynamical systems of the operator when $cd(c+d)=0$. In this case
 it is shown that the set of fixed points may be an uncountable set. For some values of parameters we
 find limit points of the trajectories depending on the initial points.
 The case $cd(c+d)\ne 0$ is considered
 in Section 4, we show that there are up to three fixed points,
 limit points of all trajectories of the operator are given. In general, we show that all trajectories converge,
 this means that the operator is regular.
 In the last section we give some biological interpretations of our results.

\section{Reduction to 2-Volterra QSO} Note that the operator (\ref{b3}) has a form of 2-Volterra QSO, but the parameters
of this operator are not related to $P_{ij,k}$. Here to make some relations with $P_{ij,k}$ we find conditions on parameters
of (\ref{b3}) rewriting it in the form (\ref{3}).
It is easy to see that $x^{(1)}+y^{(1)}+z^{(1)}=x+y+z$, to embed a vector $(x,y,z)$ in the set $S^2$ we assume $x+y+z=1$ and $x,y,z\geq 0$.

Using $x+y+z=1$,  the system (\ref{b3}) can be written as the following:

\begin{equation}\label{9}
\begin{cases}
x^{(1)}=x[(1-b)x+(1-b+d)y+(1-b)z]\\
y^{(1)}=y[(1-a-d)x+(1-a)y+(1-a+c)z]\\
z^{(1)}=z[x+(1-c)y+z]+ay+bx
\end{cases}
\end{equation}

Third equation of the system (\ref{9}) can be written as the following:
$$z^{(1)}=z[x+(1-c)y+z]+ay(x+y+z)+bx(x+y+z)$$
\begin{equation}\label{10}
 =z[(1+b)x+(1+a-c)y+z]+bx^{2}+(a+b)xy+ay^{2}.
\end{equation}
We consider (\ref{3}) for the case $m=3$:

\begin{equation}\label{11}
\left\{
\begin{array}{r}
x^{(1)}=P_{11,1}x^{2}+2P_{12.1}xy+2P_{13,1}xz+2P_{23,1}yz+P_{22,1}y^{2}+P_{33,1}z^{2}\\
y^{(1)}=P_{11,2}x^{2}+2P_{12.2}xy+2P_{13,2}xz+2P_{23,2}yz+P_{22,2}y^{2}+P_{33,2}z^{2}\\
z^{(1)}=P_{11,3}x^{2}+2P_{12.3}xy+2P_{13,3}xz+2P_{23,3}yz+P_{22,3}y^{2}+P_{33,3}z^{2}
\end{array}\right.
\end{equation}
From equation \, (\ref{10}) \, and by the systems \,(\ref{9})\, and \,(\ref{11})\, we have the following relations:
\begin{equation}\label{par}
\begin{array}{lll}
P_{11,1}=1-b, \ \ 2P_{12,1}=1-b+d, \ \ 2P_{13,1}=1-b, \ \ P_{23,1}=P_{22,1}=P_{33,1}=0 \\[2mm]
P_{22,2}=1-a, \ \ 2P_{12,2}=1-a-d, \ \ 2P_{23,2}=1-a+c, \ \ P_{11,2}=P_{13,2}=P_{33,2}=0 \\[2mm]
P_{11,3}=b, \ \  2P_{12,3}=a+b, \ \ 2P_{13,3}=1+b,  \ \ 2P_{23,3}=1+a-c,  \ \ P_{22,3}=a, \ \ P_{33,3}=1.
\end{array}
\end{equation}
\begin{pro} The operator (\ref{b3}) maps $S^2$ to itself if and only if
\begin{equation}\label{cond}
0\leq a\leq 1, \ \ 0\leq b\leq 1, \ \ -(1-a)\leq c\leq 1+a, \ \ -(1-b)\leq d\leq 1-a.
\end{equation}
Moreover, under condition (\ref{cond}) the operator is a 2-Volterra QSO.
\end{pro}
\begin{proof} The proof consists solving of simple inequalities which are obtained from conditions (\ref{4})
for $P_{ij,k}$ given by equalities (\ref{par}).
\end{proof}
\begin{rk} If $a=b, \, c=0, d=0$ then the operator (\ref{b3}) coincides with  operator (5.1)
in \cite{UR}.
By the condition (5.3) in the paper \cite{UR} and by the (\ref{par}) we have the following:
$$P_{11,1}=P_{22,2}  \Rightarrow a=b,$$
$$P_{12,1}=P_{12,2}   \Rightarrow d=0,$$
$$P_{13,1}=P_{23,2}    \Rightarrow c=0.$$
\end{rk}
\begin{rk} In the sequel of the paper we consider operator (\ref{b3}) with four parameters $a,b,c,d$ which satisfy
condition (\ref{cond}). This operator maps $S^2$ to itself and we are interested to study the behavior of the trajectory
of any initial point $(x,y,z)\in S^2$ under iterations of the operator (\ref{b3}).
\end{rk}

\section{Case $cd(c+d)=0$}

\subsection{Case $c=d=0$} In this case the operator (\ref{b3}) becomes a linear operator:
\begin{equation}\label{01}
V:
\begin{cases}
x^{(1)}=x(1-b)\\
y^{(1)}=y(1-a)\\
z^{(1)}=z+ay+bx.
\end{cases}
\end{equation}
\subsubsection{Case:} $a=b=0$. In this case the operator is $id$ map, i.e., $V(x,y,z)=(x, y, z)$.
\subsubsection{Case:} $a\ne 0, b=0$.
We have
\begin{equation}\label{u1}
\lim_{n\to \infty} V^n(x, y, z)=\lim_{n\to \infty} (x^{(n)}, y^{(n)}, z^{(n)})
=\lim_{n\to \infty} (x, (1-a)^ny, z^{(n)})=(x,0,1-x),
\end{equation}
thus in this case there are infinitely many fixed points $(x, 0, 1-x)$ and the trajectory started at any
initial point $(x, y, z)$ has the limit point $(x, 0, 1-x)$.

The case $a=0$, $b\ne 0$ is similar, and the limit point is $(0, y, 1-y)$.\\
\subsubsection{Case:} $ab\ne 0$. In this case the operator has a unique fixed point
$(0,0,1)$. Moreover, for any
initial point $(x,y,z)\in S^2$ we have
\begin{equation}\label{u2}
\lim_{n\to \infty} V^n(x,y,z)=\lim_{n\to \infty} (x^{(n)}, y^{(n)}, z^{(n)})=(0,0,1),
\end{equation}
with $x^{(n)}=(1-b)^n x, y^{(n)}=(1-a)^n y$.

\subsection{Case:} $c=0, d\ne 0$.  In this case the third coordinate of the operator has a linear form:
\begin{equation}
\begin{cases}
x^{(1)}=x(1-b+dy)\\
y^{(1)}=y(1-a-dx)\\
z^{(1)}=z+ay+bx.
\end{cases}
\end{equation}
\subsubsection{Case:} $a=b=0$. To find fixed points we solve $V(x,y,z)=(x,y,z)$, i.e.
$$x=x(1+dy), \ \ y=y(1-dx),\ \ z=z.$$
This system has infinitely many solutions, i.e., the following are the set of fixed points:
$$F_{1}=\{(x,y,z)\in S^2: y=0\},\ \ \ F_{2}=\{(x,y,z)\in S^2: x=0\}.$$
For the trajectory $(x^{(n)}, y^{(n)}, z^{(n)})$ we have
\begin{equation}\label{ns}
\begin{cases}
x^{(n+1)}=x^{(n)}(1+dy^{(n)})\\
y^{(n+1)}=y^{(n)}(1-dx^{(n)})\\
z^{(n+1)}=z^{(n)}=z,
\end{cases}
\end{equation}
i.e. the third coordinate does not depend on $n$, i.e., $z^{(n)}=z$. Now using $x^{(n)}+y^{(n)}+z^{(n)}=1$ from the first equality of
(\ref{ns}) we get $x^{(n+1)}=x^{(n)}(1+d(1-x^{(n)}-z))$, thus the behavior of $x^{(n)}$ is given by
the one-dimensional dynamical system generated by the function
$$f(x)=x(1+d(1-x-z))=-dx^2+[1+d(1-z)]x, \ \ x\in [0,1]$$ here $d$ and $z$ are considered as parameters.

Note that $f(x)$ has two fixed points $x=0$ and $x=1-z$.

Assume $d>0$ then $f'(0)=1+d(1-z)>1, z\ne 1$
(note that $z=1$ gives the fixed point $(0,0,1)$),
and $f'(1-z)=1-d(1-z)<1$. Hence $0$ is repelling and $1-z$ is attracting fixed point for $f$.
To find two-periodic points of $f(x)$ one has to solve $f(f(x))=x$. Note that the fixed points of $f$ also
solutions to $f(f(x))=x$. To find the two-periodic points different from the fixed points one has to solve
$${f(f(x))-x\over f(x)-x}=0.$$ This equation is a quadratic equation whose discriminant is $D=d^2(z-1)^2-4<0$, for any
$d\in [-1,1]$ and $z\in [0,1]$. Therefore, the function $f$ has no any two periodic point (different from fixed points)
then by Sarkovskii's theorem (see page 62 of \cite{De}) $f$ has no any periodic
(except fixed) point. For initial point $(x,y,z)$ we have $x=1-y-z\leq 1-z$, moreover,
for $d>0$ from (\ref{ns}) one can see that $x^{(n)}$ is increasing,
therefore it has a limit. Thus $1-z$ is globally attracting, i.e., for any initial point $x\in (0,1]$ we have
$\lim_{n\to\infty}x^{(n)}=1-z.$

Similarly in case $d<0$ we have $0$ is attracting and $1-z$ is repeller fixed point for $f$. Moreover,
for any initial point $x\in [0,1)$ we have
$$\lim_{n\to\infty}x^{(n)}=0.$$

Thus, if $d>0$ (resp. $d<0$) then for any initial point $(x,y,z)\in S^2$ we have
 the following
 \begin{equation}\label{u3}
\lim_{n\to \infty} V^n(x,y,z)=\lim_{n\to \infty} (x^{(n)}, y^{(n)}, z^{(n)})=(1-z,0,z) \ \ \ ({\rm resp.} \ \ (0,1-z,z)),
\end{equation}
because in this case there is no fixed points in int$S^2=\{(x,y,z)\in S^2: xyz>0\}$.

\subsubsection{Case:} $a=0, b\ne0$. Then the operator (\ref{b3}) has the following form
\begin{equation}\label{rk21}
\begin{cases}
x^{(1)}=x(1-b+dy)\\
y^{(1)}=y(1-dx)\\
z^{(1)}=z+bx.
\end{cases}
\end{equation}
It is easy to see that the operator (\ref{rk21}) has infinitely  many fixed points given by the following set
$$F_{3}=\{(x,y,z)\in S^2: x=0\}.$$ Moreover, it does not have a fixed point outside of $F_3$.

We have $z^{(1)}=z+bx\geq z$, consequently $z^{(n+1)}\geq z^{(n)}$,
i.e. $z^{(n)}$ increasing. Since $z^{(n)}\leq 1$ it has a limit.

In addition, if $d>0$ (resp. $d<0$) then $y^{(1)}=y(1-dx)\leq y$
(resp. $x^{(1)}=x(1-b+dy)\leq x$), consequently, $y^{(n)}$ (resp. $x^{(n)}$)
is decreasing sequence and it has a limit point. Hence,
for any initial point $(x,y,z)\in S^2$ there exists the limit
\begin{equation}\label{u4}
\lim_{n\to \infty} V^n(x,y,z)=\lim_{n\to \infty} (x^{(n)}, y^{(n)}, z^{(n)})=(0,\bar{y},1-\bar{y}),
\end{equation}
 because all limit points are fixed points belong to the set $F_{3}$, where $\lim_{n\to \infty} y^{(n)}=\bar{y}.$

\subsubsection{Case:}  $a\ne0, b=0$. This case is similar to the previous case the operator  (\ref{b3}) has the form
\begin{equation}\label{rk22}
\begin{cases}
x^{(1)}=x(1+dy)\\
y^{(1)}=y(1-a-dx)\\
z^{(1)}=z+ay.
\end{cases}
\end{equation}
The operator (\ref{rk22}) also has the following infinite set of fixed points
$$F_{4}=\{(x,y,z)\in S^2: y=0\}.$$
Here also $z^{(1)}=z+ay\geq z$ so $z^{(n)}$ increasing,  if $d>0$  then $x^{(1)}=x(1+dy)\geq x$, $y^{(1)}=y(1-a-dx)\leq y$ and $x^{(n)}$ increasing, $y^{(n)}$ decreasing. If $d<0$ then $x^{(1)}=x(1+dy)\leq x$ so $x^{(n)}$  decreasing sequence. Hence, for any initial point $(x,y,z)\in S^2$ we have
\begin{equation}\label{u5}
\lim_{n\to \infty} V^n(x,y,z)=\lim_{n\to \infty} (x^{(n)}, y^{(n)}, z^{(n)})=(\bar{x},0,1-\bar{x}),
  \end{equation}
because all fixed points belong to the set $F_{4}$, where $\lim_{n\to \infty} x^{(n)}=\bar{x}.$\\

\begin{rk} In limits (\ref{u4}) and (\ref{u5}) we only know existence of $\bar{y}$ and $\bar{x}$.
Of course, these values depend on the initial point $(x,y,z)$, i.e. $\bar{y}=\bar{y}(x,y,z)$ and $\bar{x}=\bar{x}(x,y,z)$. But
finding an explicit form of these values seems difficult problem.
This problem can be reduced to following non-linear recursive equations:

For finding of $z^{(n)}$ in (\ref{u4}):
\begin{equation}\label{z1}
z^{(n+2)}=-{d\over b}(z^{(n+1)})^2+{(2-b)d\over b}z^{(n+1)}z^{(n)}-{(1-b)d\over b}(z^{(n)})^2+(2-b+d)z^{(n+1)}-(1-b+d)z^{(n)},
\end{equation}
where $z^{(0)}=z$ and $z^{(1)}=z+bx$.

For finding of $z^{(n)}$ in (\ref{u5}):
\begin{equation}\label{z2}
z^{(n+2)}={d\over a}(z^{(n+1)})^2-{(2-a)d\over a}z^{(n+1)}z^{(n)}+{(1-a)d\over a}(z^{(n)})^2+(2-a-d)z^{(n+1)}-(1-a-d)z^{(n)},
\end{equation}
where $z^{(0)}=z$ and $z^{(1)}=z+ay$.
To the best of our knowledge there is no any general method to solve these recursive equations (with given initial values).
But formulas (\ref{z1}) and (\ref{z2}) are useful for using of a computer program.
\end{rk}
\subsection{Case:} $c\ne0, d=0$. In this case the system  (\ref{b3}) has the form
\begin{equation}\label{rk23}
\begin{cases}
x^{(1)}=x(1-b)\\
y^{(1)}=y(1-a+cz)\\
z^{(1)}=z(1-cy)+ay+bx
\end{cases}
\end{equation}
For the operator (\ref{rk23}) fixed points are
$$F_{5}=\left\{\begin{array}{ll}
(0,0,1) \ \ \mbox{if} \ \ b\ne 0, c<0\\[3mm]\\
\{(0,0,1), (0, 1-{a\over c}, {a\over c})\} \ \ \mbox{if} \ \ b\ne 0, c>0\\[3mm]\\
\{(x,0,1-x),(x, 1-x-{a\over c}, {a\over c})\} \ \ \mbox{if} \ \ b=0, c>0.\\[3mm]\\
\{(0,0,1), (x,0,1-x)\} \ \ \mbox{if} \ \ b=0, c<0.\\[3mm]
 \end{array}
 \right.$$

\subsubsection{Case:} $b=0$. Then $x^{(n)}=x$ for any $x\in [0,1]$, consequently
 $$y^{(n+1)}=y^{(n)}(1-a+c(1-x-y^{(n)})).$$ Hence $y^{(n)}$ is given by a dynamical system of
 $$g(y)=y(1-a+c(1-x-y))=-cy^2+(1-a+c(1-x))y,$$
 where $a, c, x$ are parameters.   This function has two fixed points: $y=0$ and $y=y_*=1-x-{a\over c}$.
If $c>0$ then  $0$ is repelling and $y_*$ is attracting fixed point for $g$.
Moreover,  for any initial point $y\in (0,1]$ we have
$$\lim_{n\to\infty}y^{(n)}=y_*.$$

Similarly in case $c<0$ we have $0$ is attracting and $y_*$ is repeller fixed point for $g$. Moreover,
for any initial point $y\in [0,1)$ we have
$$\lim_{n\to\infty}y^{(n)}=0.$$

Thus, if $c>0$ (resp. $c<0$) then for any initial point $(x,y,z)\in S^2$ we have
 the following
\begin{equation}\label{u6}
\lim_{n\to \infty} V^n(x,y,z)=\lim_{n\to \infty} (x^{(n)}, y^{(n)}, z^{(n)})=(x,y_*,1-x-y_*) \ \ \ ({\rm resp.} \ \ (x,0,1-x)).
\end{equation}

\subsubsection{Case:} $b\neq0$. Then $x^{(n)}=(1-b)^nx$ and when $c\leq a$, $(0,0,1)$ unique fixed point, $y^{(1)}=y(1-a+cz)\leq y$, i.e. the sequence $y^{(n)}$ is monotone decreasing so for any initial point  $(x,y,z)\in S^2$
$$\lim_{n\to \infty} (x^{(n)}, y^{(n)}, z^{(n)})=(0,0,1).$$
If $c>a$ then the set
 $$H=\{(x,y,z)\in S^2: z>\frac{a}{c}\}$$
is an invariant.  Indeed, if  $z>\frac{a}{c}$  then
$$1-x-y>\frac{a}{c} \Leftrightarrow cy<c-a-cx \Leftrightarrow 1-cy>1-(c-a)+cx>0 \Leftrightarrow  z^{(1)}=z(1-cy)+ay>\frac{a}{c}.$$
It is easy to see that if initial point $(x,y,z)\in H$ then the sequence $y^{(n)}$ has limit as an increasing and bounded sequence,
if $(x,y,z)\notin H$  then for $y^{(n)}$ we have two possibilities:
stays outside of $H$ and decreasingly goes to $1-{a\over c}$;
after finite steps goes inside of $H$ and increasingly converges to the same limit.

 Hence, for any initial point $(x,y,z)\in S^2$ we have
 \begin{equation}\label{u7}
 \lim_{n\to \infty} (x^{(n)}, y^{(n)}, z^{(n)})=
\begin{cases}
(0, 1-{a\over c}, {a\over c}),   \ \ \mbox{if} \ \  y>0\\[2mm]
(0,0,1),    \ \ \mbox{if}  \ \  y=0\\
\end{cases}
\end{equation}

\subsection{Case:} $c=-d\neq0$. Here the operator  (\ref{b3}) has the form
\begin{equation}\label{rk24}
\begin{cases}
x^{(1)}=x(1-b-cy)\\
y^{(1)}=y(1-a+cx+cz)\\
z^{(1)}=z(1-cy)+ay+bx
\end{cases}
\end{equation}

For the operator (\ref{rk24}) fixed points are
$$F_{6}=\left\{\begin{array}{ll}
(0,0,1) \ \ \mbox{if} \ \ c\leq a, b\ne 0 \\[3mm]\\
\{(x, 0, 1-x)\} \ \ \mbox{if} \ \ c\leq a, b= 0\\[3mm]\\
\{(0,0,1),(0, 1-{a\over c}, {a\over c})\} \ \ \mbox{if} \ \ c>a, b\ne 0 \\[3mm]\\
\{(x,0,1-x),(0, 1-{a\over c}, {a\over c}) \} \ \ \mbox{if} \ \ c>a, b=0 \\[3mm]
 \end{array}
 \right.$$
\subsubsection{Case:} $0<c\leq a$ . Then $x^{(1)}=x(1-b-cy)\leq x$, $y^{(1)}=y(1-a+cx+cz)=y(1-(a-c(x+y)))\leq y$, i.e. the sequences $x^{(n)}$ and $y^{(n)}$ are decreasing.\\
\subsubsection{Case:} $c\leq a, c<0$ . Then  $y^{(1)}=y(1-a+cx+cz)\leq y$, $z^{(1)}=z(1-cy)+ay+bx\geq z,$ i.e. the sequences $y^{(n)}$ and $z^{(n)}$ are monotone. Thus, if $b\neq0$ then for any initial point  $(x,y,z)\in S^2$ we have
\begin{equation}\label{u8}\lim_{n\to \infty} (x^{(n)}, y^{(n)}, z^{(n)})=(0,0,1)\end{equation}
 (since $(0,0,1)$ is unique fixed point).\\
If $b=0$ then for any initial point  $(x,y,z)\in S^2$
\begin{equation}\label{u9}\lim_{n\to \infty} (x^{(n)}, y^{(n)}, z^{(n)})=(\tilde{x},0,1-\tilde{x}),\end{equation}
where $\lim_{n\to \infty}{x^{(n)}}=\tilde{x}.$
\subsubsection{Case:} $c> a$. In the case  $c> a$ for any $(x,y,z)\in S^2$
we have (this is a particular case of Theorem \ref{ts} given in the next section)
\begin{equation}\label{u9}
 \lim_{n\to \infty}(x^{(n)}, y^{(n)}, z^{(n)})=
\begin{cases}
(0,0,1),   \ \ \ \mbox{if} \ \ y=0, b\neq0\\
(x,0,1-x), \ \ \mbox{if}  \ \ y=0, b=0\\
(0, 1-{a\over c}, {a\over c}), \ \ \mbox{if} \ \ y>0.
\end{cases}
\end{equation}
Thus independently on parameters $a,b,c,d$ and initial point $(x,y,z)$ the limit of trajectory exists.
Summarizing above mentioned results (i.e., (\ref{u1}), (\ref{u2}), \dots, (\ref{u9})) we get
\begin{thm}\label{t1} If $cd(c+d)=0$ then for any $(x,y,z)\in S^2$ we have
$$\lim_{n\to \infty} V^n(x, y, z)=\left\{\begin{array}{lllllllllllllllll}
(x,y,z), \ \ \ \mbox{if} \ \ a=b=c=d=0\\[2mm]
(x,0,1-x), \ \ \ \mbox{if} \ \ a\ne 0, b=c=d=0\\[2mm]
(0,y,1-y), \ \ \ \mbox{if} \ \ a=c=d=0, b\ne 0\\[2mm]
(0,0,1), \ \ \ \mbox{if} \ \ ab\ne 0, \, c=d=0\\[2mm]
(1-z,0,z), \ \ \ \mbox{if} \ \ a=b=c=0, d>0\\[2mm]
(0,1-z,z), \ \ \ \mbox{if} \ \ a=b=c=0, d<0\\[2mm]
(0,\bar y,1-\bar y), \ \ \ \mbox{if} \ \ a=c=0, bd\ne 0\\[2mm]
(\bar x, 0, 1-\bar x), \ \ \ \mbox{if} \ \ b=c=0, ad\ne 0\\[2mm]
(x, 1-x-{a\over c}, {a\over c}), \ \ \ \mbox{if} \ \ c>0, b=d=0\\[2mm]
(x,0,1-x), \ \ \ \mbox{if} \ \ b=d=0, c<0\\[2mm]
(0,0,1), \ \ \ \mbox{if} \ \ b\ne 0, \, c\leq a, d=0, \ \ \mbox{or} \ \ y=0\\[2mm]
(0, 1-{a\over c}, {a\over c}) \ \ \ \mbox{if} \ \ b\ne 0, \, c>a, d=0, \ \ y>0\\[2mm]
(0,0,1), \ \ \ \mbox{if} \ \ b\ne 0, \, c=-d\ne 0, 0<c\leq a,\\[2mm]
(\tilde{x}, 0, 1-\tilde{x}), \ \ \ \mbox{if} \ \ b=0, \, c=-d\ne 0, c\leq a,\\[2mm]
(0,0,1), \ \ \ \mbox{if} \ \ b\ne 0, \, c=-d, a<c, y=0,\\[2mm]
(x,0,1-x), \ \ \ \mbox{if} \ \ b=0, \, c=-d, a<c, y=0,\\[2mm]
(0, 1-{a\over c}, {a\over c}) \ \ \ \mbox{if} \ \ a<c \ \ y>0,
\end{array}\right.$$
where $\bar x$, $\bar y$ and $\tilde x$ are some functions of the initial point $(x,y,z)$.
\end{thm}

\section{Case $cd(c+d)\ne 0$}

\subsection{Fixed points of the operator (\ref{b3})}

\begin{defn}\cite{De}. A fixed point $p$ for $F: R^{m}\rightarrow R^{m}$ is called \emph{hyperbolic} if the Jacobian matrix $\textbf{J}=\textbf{J}_F$ of the map $F$ at the point $p$ has no eigenvalues on the unit circle.

There are three types of hyperbolic fixed points:

(1) $p$ is an attracting fixed point if all of the eigenvalues of $\textbf{J}(p)$ are less than one in absolute value.

(2) $p$ is an repelling fixed point if all of the eigenvalues of $\textbf{J}(p)$ are greater than one in absolute value.

(3) $p$ is a saddle point otherwise.
\end{defn}

To find fixed points of operator $V$ given by (\ref{b3}) we have to solve $V(x)=x$, i.e.
\begin{equation}\label{fp}
\begin{cases}
x=x(1-b+dy)\\
y=y(1-a-dx+cz)\\
z=z(1-cy)+ay+bx.
\end{cases}
\end{equation}

\begin{lemma}\label{lf} The fixed points of the operator (\ref{b3}) are
$$\bar{\lambda}_{0}=(0,0,1), \ \ \bar{\lambda}_{1}=\left(0,1-\frac{a}{c},\frac{a}{c}) \ \ (c\geq a, c\ne 0\right), $$  and  $$\bar{\lambda}_{2}=\left(\frac{cd-ad-bc}{d(c+d)},\frac{b}{d},\frac{a-b+d}{c+d}\right) $$
where there are conditions to the parameters, in other words we consider all coordinates of $\bar{\lambda}_{2}$ between 0 and 1:
 \begin{equation}\label{conpar}
 d>0,\ \ d\geq b, \ \ a-b\leq c, \ \ c\geq0, \ \ cd-ad-bc \geq 0.
  \end{equation}
\end{lemma}

\begin{proof} Consider the following cases

\textbf{Case} $x=0, y=0$. It is easy to see from (\ref{fp}) that $(0,0,1)$ is a fixed point.

\textbf{Case} $x=0, y\neq 0$. From  (\ref{fp}) we get system of equations
 $$x=0, \ \ y(1-a+cz)=y, \ \ z(1-cy)+ay=z,$$
 which has a unique solution $\left(0, 1-\frac{a}{c}, \frac{a}{c}\right)$ for $c>0, c\geq a.$

\textbf{Case} $x\neq 0, y=0$. In this case we do not have fixed point, because from the
first equation of (\ref{fp}) we get $x=x(1-b)$ which has no non-zero solution for $b\ne 0$.

\textbf{Case} $x\neq 0, y\neq 0$. In this case the system (\ref{fp}) is reduced to the system
$$1=1-b+dy \ \ 1=1-a-dx+cz \ \ z=z(1-cy)+ay+bx.$$
from which we get the third fixed point $\left(\frac{cd-ad-bc}{d(c+d)},\frac{b}{d},\frac{a-b+d}{c+d}\right).$
\end{proof}

By using $x+y+z=1$ in (\ref{b3}),  we obtain the following mapping:

\begin{equation}\label{b2}
W:
\begin{cases}
x'=x(1-b+dy)\\
y'=y(1-a+c-(c+d)x-cy)
\end{cases}  \ \ {\rm where} \ \ x+y\leq1.
\end{equation}

For the operator $W$ there are three fixed points with conditions (\ref{conpar}):
$$\lambda_{0}=(0,0), \ \ \lambda_{1}=\left(0,\frac{c-a}{c}\right) \ \ (c>0, c\geq a),  \ \ \text{and }\ \ \lambda_{2}=\left(\frac{cd-ad-bc}{d(c+d)},\frac{b}{d}\right).$$

\begin{pro}\label{pp1} The following relations are true:
\begin{itemize}
\item[(1)]
$$\lambda_{0}=\left\{\begin{array}{lll}
{\rm nonhyperbolic}, \ \ {\rm if} \ \  b=0 \ \ {\rm or} \ \ a=c\\[2mm]
{\rm attractive}, \ \ \ \ {\rm if} \ \  a>c\\[2mm]
{\rm saddle}, \ \ \  \ \ \ {\rm if}  \ \ a<c
\end{array}\right.$$

\item[(2)] $$\lambda_{1}=\left\{\begin{array}{lll}
{\rm nonhyperbolic}, \ \ {\rm if} \ \ a=c \ \ {\rm or} \ \ b=d(1-\frac{a}{c}) \\[2mm]
{\rm attractive}, \ \ {\rm if} \ \  b>d(1-\frac{a}{c})  \\[2mm]
{\rm saddle}, \ \ {\rm if}  \ \  b<d(1-\frac{a}{c})
\end{array}\right.$$

\item[(3)] $$\lambda_{2}=\left\{
\begin{array}{ll}
{\rm nonhyperbolic}, \ \ {\rm if} \ \ b=0 \ \  {\rm or} \ \ c\leq a \ \  {\rm or} \ \ cd-ad-bc=0 \ \  {\rm or}  \ \ b=d\\[2mm]
{\rm attractive}, \ \ {\rm if} \ \  c>0, c>a, d>b>0, cd-ad-bc>0
\end{array}\right.$$
\end{itemize}
\end{pro}
\begin{proof}

(1) Note that the Jacobian of the system (\ref{b2}) has the form

\begin{equation}\label{jac}
\textbf{J}=\begin{bmatrix}
1-b+dy & dx\\
-(c+d)y & 1-a+c-(c+d)x-2cy
\end{bmatrix}
\end{equation}

The Jacobian at the fixed point $\lambda_{0}=(0,0)$ has the form

\begin{center}
$\textbf{J}(\lambda_{0})=\begin{bmatrix}
1-b & 0 \\
0 & 1-a+c
\end{bmatrix}$
\end{center}
Clearly, if $b=0$ or $a=c$  then the eigenvalues $\mu_{1}$ or $\mu_{2}$ of the matrix $\textbf{J}(\lambda_{0})$ are equal to one, i.e., $\lambda_{0}$ is nonhyperbolic and in all other cases $\lambda_{0}$ is hyperbolic. \\
If $b>0, a\neq c$ then one eigenvalue $\mu_{1}=1-b<1$ is correct for all $b$.   If $a>c$ and $a<c$ then $|\mu_{2}|<1$ and $|\mu_{2}|>1$ respectively. Hence,  $\lambda_{0}$ is an attracting and saddle point when $a>c$ and $a<c$ respectively, repelling case does not occur.

(2) At the point $\lambda_{1}$ the Jacobian is:

\begin{center}
$\textbf{J}(\lambda_{1})=\begin{bmatrix}
1-b+\frac{(c-a)d}{c} & 0\\
-\frac{(c+d)(c-a)}{c} & 1+a-c
\end{bmatrix}$
\end{center}
It is easy to see that if $a=c$  or $b=d(1-\frac{a}{c})$ then at least one eigenvalue of the matrix $\textbf{J}(\lambda_{1})$  is equal to one, so $\lambda_{1}$ is nonhyperbolic.\\
If $a\neq c$ then by the condition $c\geq a$  one eigenvalue $\mu_{1}=1-c+a$ is always less than one.  Moreover, if $b>d(1-\frac{a}{c})$ then   $\lambda_{1}$ is an attractive, and if $b<d(1-\frac{a}{c})$ then $\lambda_{1}$ is a saddle point. Here also repelling case does not exists.

(3) Next we consider the Jacobian at the fixed point $\lambda_{2}$:

\begin{center}
$\textbf{J}(\lambda_{2})=\begin{bmatrix}
1 & \frac{cd-ad-bc}{c+d}\\
-\frac{(c+d)b}{d} & \frac{d-bc}{d}
\end{bmatrix}$
\end{center}
$\circ$  If $b=0$ or $cd-ad-bc=0$ then $\lambda_{2}$ is nonhyperbolic. \\
$\circ$  If $c=0$ then $a=0$ and $\lambda_{2}=(0,\frac{b}{d})$ is nonhyperbolic again.\\
$\circ$ If $b=d$ then by the condition $cd-ad-bc\geq0$ we  have $a=0$, it means   $\lambda_{2}$ is nonhyperbolic.\\
$\circ$ If $0<c\leq a$ then by the condition $cd-ad-bc\geq0$ we have $c=a$ and $b=0$ and $\lambda_{2}$ is nonhyperbolic again.\\
$\circ$  If $d>b>0, c>a, c>0$ then the eigenvalues of the matrix $\textbf{J}(\lambda_{2})$ are:
$$\mu_{1,2}=\frac{2d-bc\pm\sqrt{D}}{2d},$$
where
$$D=b(bc^{2}-4d(cd-ad-bc)).$$
 We study this eigenvalues:\\
If $b\geq \frac{4d^{2}(c-a)}{c(c+4d)}$ then
$$D=b^{2}c^{2}-4bd(cd-ad-bc)\geq 0$$ and since $D< b^{2}c^{2}$ we have $\mu_{1,2}<1$. Furthermore, we have to check the condition  $\mu_{1,2}>-1$:
If $d>\frac{bc(b+2)}{4+bc-ab}$ then $\lambda_{2}$ is a attractive fixed point\\
If $d<\frac{bc(b+2)}{4+bc-ab}$ then $\lambda_{2}$ is a saddle fixed point\\
If $d=\frac{bc(b+2)}{4+bc-ab}$ then $\lambda_{2}$ is nonhyperbolic fixed point.\\
But by the conditions (\ref{cond})
$$\frac{bc(b+2)}{4+bc-ab}-d=\frac{b^{2}c+2bc-4d-bcd+abd}{4+bc-ab}<\frac{d^{2}c+2dc-4d-bcd+cbd}{4+bc-ab}=$$
$$=\frac{d(cd+2c-4)}{4+bc-ab}<\frac{d((1+a)(1-a)+2(1+a)-4)}{4+bc-ab}=-\frac{d(a-1)^{2}}{4+bc-ab}<0.$$
It follows that $$d>\frac{bc(b+2)}{4+bc-ab}.$$
If $b< \frac{4d^{2}(c-a)}{c(c+4d)}$ then $D<0$ and we obtain the complex eigenvalues. \\
If $a<\frac{c(2d^2-b^{2}c)}{2d^{2}}$ then $|\mu_{1}|=|\mu_{2}|=\frac{\sqrt{(2d-bc)^{2}+D}}{2d}<1$ and $\lambda_{2}$ is attracting fixed point\\
If $a>\frac{c(2d^2-b^{2}c)}{2d^{2}}$ then $|\mu_{1}|=|\mu_{2}|>1$ and $\lambda_{2}$ is repelling fixed point\\
If $a=\frac{c(2d^2-b^{2}c)}{2d^{2}}$ then $|\mu_{1}|=|\mu_{2}|=1$ and $\lambda_{2}$ is nonhyperbolic fixed point.
By the conditions to the parameters again we have:
$$\frac{c(2d^{2}-b^{2}c)}{2d^{2}}-a=\frac{2cd^2-b^{2}c^2-2ad^2}{2d^{2}}=\frac{2cd^2-2ad^2-2bcd+2bcd-b^{2}c^2}{2d^{2}}=$$
$$=\frac{2d(cd-ad-bc)+bc(2d-bc)}{2d^{2}}>\frac{2d(cd-ad-bc)+bc(2b-bc)}{2d^{2}}>0  $$
( since  $cd-ad-bc>0, c\leq 2$).
It follows that $$a<\frac{c(2d^2-b^{2}c)}{2d^{2}}.$$
Consequently, for $cd-ad-bc>0, c>a, c>0, d>b>0$ the fixed point $\lambda_{2}$ is attracting.
\end{proof}
\subsection{\bf The limit points}

 We consider the operator (\ref{b2}) with initial point $\lambda^{(0)}=(x^{(0)},y^{(0)})$ and
 will study when some of fixed points $\lambda_{0}, \lambda_{1}$ and $\lambda_{2}$ is a limit point.

If $x^{(0)}=y^{(0)}=0$ then the limit point is $\lambda_{0}$.

The invariant sets with respect to $W$, (\ref{b2}),
 are $$M_{1}=\{\lambda=(x,y): x=0\}, \ \  M_{2}=\{\lambda=(x,y): y=0\},$$
  i.e., $W(M_{i})\subset M_{i}, i=1,2.$

Consider restriction of the operator (\ref{b2}) on invariant sets:

{\it Case $M_1$}. In this case the restriction is
$$y^{(1)}=y(1-a+c-cy)=\varphi(y), \ \ y\in [0,1].$$
Note that $\varphi(y)$ has a unique fixed point $y=0$ if $a>c$ and two fixed points
$y=0$ , $y=1-{a\over c}$ if $c\geq a$.
Since $\varphi'(0)=1-a+c$  the fixed point
$0$ is attracting iff $c<a$, non-hyperbolic if $c=a$ and repelling if $c>a$.
Similarly, by $\varphi'(1-{a\over c})=1+a-c$ we have that $1-{a\over c}$ is
attracting if $c>a$ and non-hyperbolic if $c=a$. Moreover, for $a\geq c$ the sequence $y^{(n)}=\varphi^n(y^{(0)})$
is a decreasing.  Thus for any $y^{(0)}\in [0,1]$ we have
$$\lim_{n\to\infty}y^{(n)}=\left\{\begin{array}{ll}
0, \ \ \mbox{if} \ \ a\geq c\\[2mm]
1-{a\over c}, \ \ \mbox{if} \ \ a<c.
\end{array}
\right.$$

{\it Case $M_2$}. In this case the restriction is
$$x^{(1)}=(1-b)x, \ \ x\in [0,1].$$
It is easy to see that
$$\lim_{n\to\infty}x^{(n)}=\lim_{n\to\infty}(1-b)^nx=\left\{\begin{array}{ll}
0, \ \ \mbox{if} \ \ b>0, \ \ or \ \ x=0,\\[2mm]
x, \ \ \mbox{if} \ \ b=0.
\end{array}
\right.$$

Thus behavior of trajectories are clear if the initial point is taken on invariant sets $M_i$, $i=1,2$.

We assume now that  $x^{(0)}\neq0, y^{(0)}\neq0.$

\emph{\textbf{Case 1.}} If $c<a$ then $\lambda_{1}$ and $\lambda_{2}$ do not
belong to the simplex (because they have negative coordinates), so $\lambda_{0}$ is
a unique  attracting fixed point.
We show that the fixed point is globally attracting, i.e. $\forall\lambda^{(0)}\in S^{2}$
one has $$\lim\limits_{n\rightarrow\infty}{\lambda}^{(n)}=\bar{\lambda}_{0}.$$
It is easy to see that
$$z^{(1)}=z^{(0)}(1-cy^{(0)})+ay^{(0)}+bx^{(0)}=z^{(0)}-y^{(0)}(cz^{(0)}-a)+bx^{(0)}$$
$$\geq z^{(0)}-y^{(0)}(cz^{(0)}-a)\geq z^{(0)} \ \ (\text{since}\ \ cz^{(0)}-a\leq0).$$
Consequently, the sequence $z^{(n)}$ is increasing and bounded, i.e., it has a limit:
$$\lim\limits_{n\rightarrow\infty}z^{(n)}=\bar{z}\geq0,$$
Next, we consider two cases:

1.1. If $d<0$ then $x^{(1)}=x^{(0)}(1-b+dy^{(0)})\leq x^{(0)}(1-b)\leq x^{(0)}$,
by iterating we get that $x^{(n+1)}\leq x^{(n)},  \forall n\in N$ and
there exists the limit $$\lim\limits_{n\rightarrow\infty}x^{(n)}=\bar{x}.$$
 By  $x^{(n)}+y^{(n)}+z^{(n)}=1$ it follows that there exists the limit of sequence $y^{(n)}$, $$\lim\limits_{n\rightarrow\infty}y^{(n)}=\bar{y}.$$

1.2. If $d\geq0$ then
$$y^{(1)}=y^{(0)}(1-a-dx^{(0)}+cz^{(0)})\leq y^{(0)}(1-a-dx^{(0)}+c)\leq y^{(0)}(1-dx^{(0)})\leq y^{(0)}$$
and similarly, in this case there exist limits of sequences $x^{(n)}, y^{(n)}$,  i.e.,
$$\lim\limits_{n\rightarrow\infty}\lambda^{(n)}=\bar{\lambda}=(\bar{x},\bar{y},\bar{z}).$$
Since a limit point must be a fixed point and in this case we have a unique fixed point on the simplex, we get  $$\lim\limits_{n\rightarrow\infty}\lambda^{(n)}=\bar{\lambda}_{0}=(0,0,1).$$
 Thus, $\lambda_{0}$ is a globally attracting fixed point.

\emph{\textbf{Case 2.}} If $c>a, d<b$  then  $x^{(1)}=x^{(0)}(1-b+dy^{(0)})\leq x^{(0)}(1-b+d)<x^{(0)}$
and we check the monotonicity of the sequence $y^{(n)}$.
In this case the following lemma is useful:

\begin{lemma}\label{l2} If $c>a, d<b$ for the operator (\ref{b3}) we have

 a) If $c+d\geq0$ then the set $\mathcal E=\{(x,y,z):z\geq\frac{a+dx}{c}\}$ is an invariant.

 b) If $c+d<0$ then the set $F=\{(x,y,z):z<\frac{a+dx}{c}\}$ is an invariant.
\end{lemma}
\begin{proof}  a) For any point $(x,y,z)\in \mathcal E$ we have
$$z\geq\frac{a+dx}{c}\Leftrightarrow 1-x-y\geq\frac{a+dx}{c} \Leftrightarrow cy\leq c-a-(c+d)x.$$
 Then $1-cy\geq 1-(c-a-(c+d)x)>0$. Next we show that
$z^{(1)}\geq\frac{a+dx^{(1)}}{c}$, i.e.
$$z(1-cy)+ay+bx\geq \frac{a+dx(1-b+dy)}{c}$$ which is equivalent to
$$z\geq \frac{a+dx-bdx+d^2xy-acy-bcx}{c(1-cy)}
 \Leftrightarrow \frac{a+dx}{c}\geq \frac{a+dx-bdx+d^2xy-acy-bcx}{c(1-cy)}
 $$ $$  \Leftrightarrow (c+d)(b-dy)\geq0$$ the last inequality is true,
 since $b-dy>0, c+d \geq 0$. Hence,  $z^{(1)}\geq\frac{a+dx^{(1)}}{c}$;

  b) If $(x,y,z)\in F$ then $z<\frac{a+dx}{c}$ and we have to show that $z^{(1)}<\frac{a+dx^{(1)}}{c}$, i.e.
  $$z(1-cy)<\frac{a+dx-bdx+d^2xy-acy-bcx}{c}$$

 \emph{ i)} If $1-cy>0$ then
 $$z<\frac{a+dx-bdx+d^2xy-acy-bcx}{c(1-cy)} \Leftrightarrow
 \frac{a+dx}{c}< \frac{a+dx-bdx+d^2xy-acy-bcx}{c(1-cy)}$$
 from this we have the inequality $(c+d)(b-dy)<0$ which is true;

\emph{  ii)} If $1-cy<0$ then
$$z>\frac{a+dx-bdx+d^2xy-acy-bcx}{c(1-cy)} \Leftrightarrow
\frac{a+dx}{c}> \frac{a+dx-bdx+d^2xy-acy-bcx}{c(1-cy)} $$
$$\Leftrightarrow \frac{dy(c+d)-b(c+d)}{(1-cy)}<0.$$
Hence we have again correct inequality  $(c+d)(dy-b)>0$ (since $c+d<0, dy-b<0$).
Thus,  $z^{(1)}<\frac{a+dx^{(1)}}{c}.$
\end{proof}
  Now we continue the proof of monotonicity of the sequence $y^{(n)}$.
  Using by Lemma \ref{l2} we have that
     if an initial point $\lambda^{0}=(x^{0},y^{0}, z^{0})$ taken from the set $\mathcal E$
     and $c+d\geq0$ then the sequence $y^{(n)}$ is increasing (since
     $y^{(n+1)}=y^{(n)}(1-a-dx^{(n)}+cz^{(n)})\geq y^{(n)}, \forall n\in N$), it has the limit and we obtain
$$\lim\limits_{n\rightarrow\infty}\lambda^{(n)}=\bar{\lambda}_{1},$$
because $\lambda_{1}$ is the unique attracting fixed point in $\mathcal E$.

If $\lambda^{(0)}\notin \mathcal E$ then we consider two cases:

(i)  If $\lambda^{(n)}\notin \mathcal E, \forall n\in N$ then $y^{(n)}$ is
decreasing (since  $y^{(n+1)}<y^{(n)}(1-a-dx+cz^{(n)})<y^{(n)}, \forall n\in N$);

(ii) If $\lambda^{(k)}\in \mathcal E,$ for some $k\in N$ then $y^{(n)}$ is increasing for $n>k.$

Similarly, if an initial point $\lambda^{(0)}=(x^{(0)},y^{(0)}, z^{(0)})$ taken from
the set $F$ and $c+d<0$ then the sequence $y^{(n)}$ is decreasing and for the case
$\lambda^{(n)}\notin F$ the sequence $y^{(n)}$ has limit point again.
Consequently,
if $y^{(0)}=0$ then $$\lim\limits_{n\rightarrow\infty}\lambda^{(n)}=\lambda_{0}.$$
Thus, $\lambda_{0}$ is a saddle fixed point  and $\lambda_{1}$ is globally attracting fixed point.

\emph{\textbf{Case 3.}} $c>a, d>b>0.$ In this case numerical analysis show
that the coordinates of the vector $\lambda^{(n)}$ are not monotone, so it
is not easy to see the limit properties of the trajectory.
Therefore we study these limits numerically for concrete values of parameters.

3.1. $a=1/4, b=1/2, c=1, d=3/4.$ Then by the system (\ref{b2}) we get
\begin{equation}\label{bb2}
\begin{cases}
x'=0.5x+0.75xy\\
y'=1.75y-1.75xy-y^{2}
\end{cases}
\end{equation}

For this system $\lambda_{2}=(\frac{1}{21};\frac{2}{3};\frac{2}{7})\approx(0.0476;0.6667;0.2857).$
Next we divide the simplex to the four parts as following (Fig.1):
\begin{figure}
\includegraphics[width=7cm]{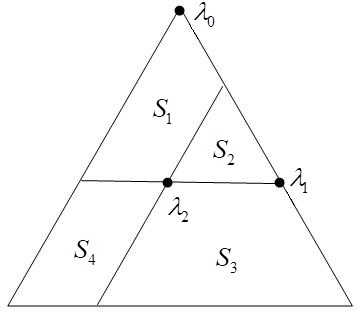}
\caption{}\label{f3}
\end{figure}

By using Wolfram Mathematica 7.0 we find limit points of initial
points from $S_{1},S_{2}, S_{3}$ and $S_{4}$ respectively (Fig.2).

i) If $(x^{(0)}, y^{(0)})=(0.1,0.6)\in S_{1}$ then after 100 iterations we get

   $$ (x^{(100)}, y^{(100)})=(0.0476346, 0.666636).$$

 ii) If $(x^{(0)}, y^{(0)})=(0.02,0.68)\in S_{2}$ then

   $$(x^{(100)}, y^{(100)})=(0.0475566, 0.666789).$$

iii) If $(x^{(0)}, y^{(0)})=(0.05,0.68)\in S_{3}$ then

   $$(x^{(100)}, y^{(100)})= (0.0476215, 0.666662).$$

 iv) If $(x^{(0)}, y^{(0)})=(0.07,0.66)\in S_{4}$ then

   $$(x^{(100)}, y^{(100)})=(0.0476302, 0.666645).$$

\begin{figure}
\includegraphics[width=10cm]{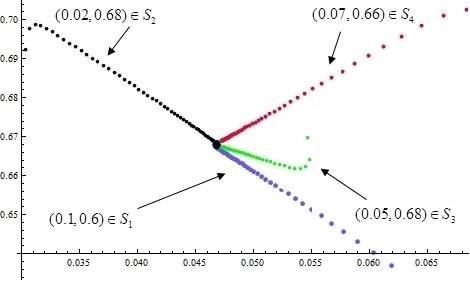}
\caption{blue:$(0.1,0.6)\in S_{1}$; black: $(0.02,0.68)\in S_{2}$;
green: $(0.05,0.68)\in S_{3}$; red: $(0.07,0.66)\in S_{4}$.  }\label{fig}
\end{figure}

3.2. $a=1/6, b=1/3, c=4/3, d=2/3.$ Then by the system (\ref{b2}) we have
\begin{equation}\label{bb3}
\begin{cases}
x'=\frac{2}{3}x+\frac{2}{3}xy\\
y'=\frac{13}{6}y-2xy-\frac{4}{3}y^{2}
\end{cases}
\end{equation}

For this system $\lambda_{2}=(0.25;0.5;0.25)$ and we have (Fig.3)

i) If $(0.26,0.48)\in S_{1}$ then
   $$(x^{(100)}, y^{(100)})=(0.25,0.5).$$

ii) If $(0.22,0.52)\in S_{2}$ then
   $$(x^{(100)}, y^{(100)})=(0.25,0.5).$$

iii) If $(0.24,0.52)\in S_{3}$ then
   $$(x^{(100)}, y^{(100)})=(0.25,0.5).$$

 iv) If $(0.29,0.48)\in S_{4}$ then
   $$(x^{(100)}, y^{(100)})=(0.25,0.5).$$

\begin{figure}
\includegraphics[width=10cm]{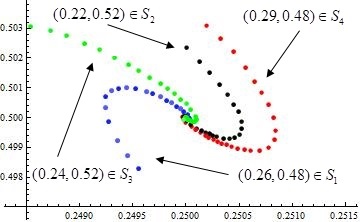}
\caption{blue: $(0.26,0.48)\in S_{1}$; black: $(0.22,0.52)\in S_{2}$;
green: $(0.24,0.52)\in S_{3}$; red: $(0.29,0.48)\in S_{4}$. }\label{fig}
\end{figure}

To sum up  for the initial point $\lambda^{(0)}=(x^{(0)},y^{(0)}, z^{(0)})\in intS^2$  we have
$$\lim\limits_{n\rightarrow\infty}\lambda^{(n)}=\bar{\lambda}_{2}.$$

We get phase portraits of the trajectories of (\ref{b3}) shown in Fig.4, Fig.5 and Fig.6.

\smallskip
\begin{figure}
\includegraphics[width=8cm]{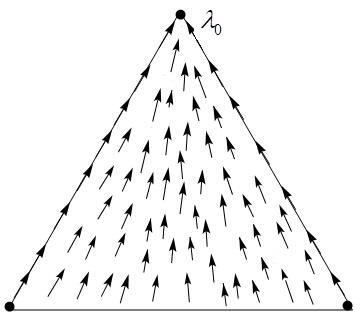}
\caption{If  $c<a$}\label{f1}
\end{figure}
\begin{figure}
\includegraphics[width=8cm]{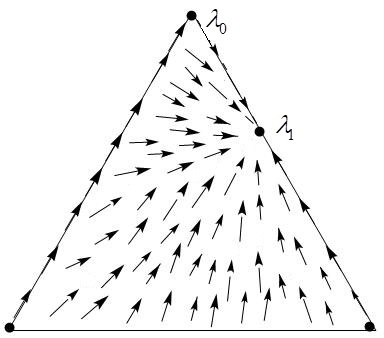}
\caption{If $c>a, d<b$}\label{f2}
\end{figure}
\begin{figure}
\includegraphics[width=8cm]{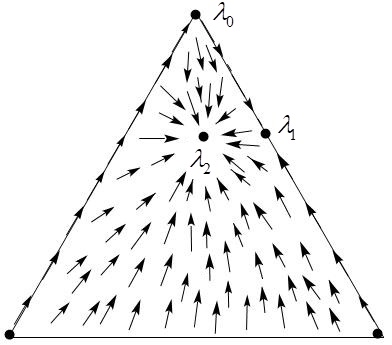}
\caption{If $c>a, d>b$}\label{f2}
\end{figure}
Summarizing we obtain the following:

\begin{thm}\label{ts} Let $cd(c+d)\ne 0$ and $\lambda^{(0)}=(x^{(0)},y^{(0)},z^{(0)})$ be an initial point.
Then  the following three cases hold

(1) If $c<a$
 then  $$\lim\limits_{n\rightarrow\infty}\lambda^{(n)}=\bar{\lambda}_{0}=(0,0,1),$$

(2) If $c>a, d<b$ then
\begin{equation}
\lim\limits_{n\rightarrow\infty}\lambda^{(n)}=
\begin{cases}
(x^{(0)},0,1-x^{(0)}) \hspace{22mm} if \hspace{20mm}y^{(0)}=0, b=0\\
\bar{\lambda}_{0} \hspace{45mm}  if \hspace{20mm}y^{(0)}=0,b\neq 0 \\
\bar{\lambda}_{1} \hspace{45mm}  if \hspace{20mm}y^{(0)}>0
\end{cases}
\end{equation}
(3) If   $c>a, d>b>0$   then
\begin{equation}\label{oo}
\lim\limits_{n\rightarrow\infty}\lambda^{(n)}=
\begin{cases}
\bar{\lambda}_{2}=\left(\frac{cd-ad-bc}{d(c+d)},\frac{b}{d},\frac{a-b+d}{c+d} \right) \hspace{5mm} if \hspace{10mm}x^{(0)}\neq0, y^{(0)}\neq0\\
\bar{\lambda}_{0} \hspace{45mm}  if \hspace{20mm}y^{(0)}=0 \\
\bar{\lambda}_{1} \hspace{45mm}  if \hspace{20mm}x^{(0)}=0
\end{cases}
\end{equation}
\end{thm}
\smallskip

\begin{rk} We were able to prove the first line of the formula (\ref{oo}) only for a small 
neighborhood of $\bar{\lambda}_{2}$, because this fixed point is an attracting point (see Proposition \ref{pp1}).
\end{rk}

\section{Biological interpretations}

The dynamical systems considered in this paper are interesting because
they are higher dimensional and such dynamical
systems are important, but there are relatively few
dynamical phenomena that are currently understood \cite{De}, \cite{GMR}.

In \cite{Brit} for the  continuous-time case 
the steady (stable) states of the system of equations (\ref{1}) are found.
Usually in models of ecosystems species at
alternate levels in the food chain benefit from an increase in nutrient supply.
It is known that if the supply of the nutrient which limits
phytoplankton growth is increased, it is not the phytoplankton but the
predatory zooplankton that benefit. 
 
The results formulated in previous sections  have the following biological interpretations:

Let $\lambda^{(0)}=(x^{(0)},y^{(0)},z^{(0)})\in S^2$ be an initial state, i.e., $\lambda^{(0)}$
is the probability distribution on the set
$\{N, P, Z\}$ of ecosystem.

Assume the trajectory $\lambda^{(m)}$ of this point has a limit
$\lambda^*=(x^*, y^*, z^*)$ (equilibrium state)  this means that
the future of the system is stable: each $N, P, Z$ survives with probability $x^*, y^*, z^*$ respectively.
For example, the nitrogen, $N$, of the system will disappear if its probability $x^*$ is zero.

Each fixed point of the operator (\ref{b3}) is an equilibrium state and
Theorem \ref{t1} gives that system may have a continuum set (but in under conditions of Theorem \ref{ts} it has up to three)
of equilibrium states, the system
stays in a neighborhood of one of the equilibrium states (stable fixed point), which depends on
the initial state.
 
Moreover, as in continuous  time (\cite{Brit}) and in discrete time (Theorem \ref{t1} and Theorem \ref{ts})
 for $c <a$, the only steady state is the
trivial one $\bar\lambda_0=(0,0,1)$. Plankton levels are very low. As $c$ increases past $a$,
 $\bar\lambda_0$ loses its
stability as an eigenvalue passes through zero, and $\bar\lambda_1$ becomes stable. When $c$
increases still further, there is a second critical value at $ b<d(1-\frac{a}{c})$ (Proposition \ref{pp1})
where $\bar\lambda_1$ loses its stability as an eigenvalue passes through zero, and $\bar\lambda_2$ becomes stable. 

 \section*{Acknowledgements}

 This work was partially supported by  Agencia Estatal de Investigaci\'on (Spain),
grant MTM2016-79661-P (European FEDER support included, UE).

\end{document}